\documentclass{article}

\usepackage{amssymb}
\usepackage{amsmath}
\usepackage{amsthm}
\usepackage[mathscr]{eucal}

\usepackage[active]{srcltx}

\newtheorem{thrm}{Theorem}[section]
\newtheorem{lemma}[thrm]{Lemma}
\newtheorem{prop}[thrm]{Proposition}

\newtheorem*{main*}{Main Theorem}

\theoremstyle{definition}
\newtheorem{defn}[thrm]{Definition}

\theoremstyle{remark}

\newtheorem*{question*}{Question}

\numberwithin{equation}{section}

\newcommand{\dbar}{$\bar{\partial}$}
\newcommand{\mdbar}{\bar{\partial}}

\newcommand{\Lb}{\overline{L}}

\newcommand{\omegab}{\bar{\omega}}

\newcommand{\lrc}{\mathcal{C}}

\newcommand{\lrs}{\mathcal{S}}

\makeatletter
\newcommand{\opnorm}{\@ifstar\@opnorms\@opnorm}
\newcommand{\@opnorms}[1]{%
    \left|\mkern-1.5mu\left|\mkern-1.5mu\left|
    #1
    \right|\mkern-1.5mu\right|\mkern-1.5mu\right|
}
\newcommand{\@opnorm}[2][]{%
    \mathopen{#1|\mkern-1.5mu#1|\mkern-1.5mu#1|}
    #2
    \mathclose{#1|\mkern-1.5mu#1|\mkern-1.5mu#1|}
} \makeatother

\begin{document}

    \title
    {Exact regularity of the 
    $\mdbar$-problem 
  with dependence on the 
   $\mdbar_b$-problem on weakly pseudoconvex
          domains in $\mathbb{C}^2$.}

    \author
    {Dariush Ehsani}
    
    \date{}

 


    \maketitle

    \bibliographystyle{plain}

\section{Introduction}

We investigate the regularity of solutions, $u$, to the
\dbar-equation $\mdbar u =f$, for \dbar-closed $(0,1)$-forms $f$
on smoothly bounded weakly pseudoconvex domains
$\Omega\subset\mathbb{C}^2$.  Regularity of the forms and
functions are measured in terms of Sobolev norms: we denote by
$W^s(\Omega)$, respectively $W^s_{(0,1)}(\Omega)$, the space of
functions, respectively $(0,1)$-forms, whose derivatives of order
$\le s$ are in $L^2(\Omega)$.  In the case of smoothly bounded
strictly pseudoconvex domains, the canonical solution (the
solution of minimal $L^2$ norm) can be shown to provide a solution
operator which preserves the Sobolev spaces, $W^s(\Omega)$ for all
$s\ge 0$; estimates for the canonical solution are due to Kohn
(see \cite{K63} and \cite{K64}).  This is not the case in the
situation of smoothly bounded weakly pseudoconvex domains as shown
by Barrett in \cite{Ba92}.  And it is not just a loss of
derivatives which takes place; Christ has shown that the canonical
solution may not even be in $C^{\infty}(\overline{\Omega})$ even
if the data form $f$ is in $C^{\infty}_{(0,1)}(\overline{\Omega})$
\cite{Ch96}.

On the other hand, using weighted Sobolev spaces, Kohn showed that
for any given $s\ge 0$, there exists a 
 weight, $\phi$, and a
solution operator, $K_{s,\phi}$,
 (which depends on the weight as
well as level of the Sobolev norm) such that $K_{s,\phi}:
W^k(\Omega)\rightarrow W^k(\Omega)$ for all $k\le s$ and such that
$\mdbar\circ K_{s,\phi}=I$ when restricted to \dbar-closed forms
\cite{K73}.  These operators can then be used to construct a
solution operator which maps
$C^{\infty}_{(0,1)}(\overline{\Omega})$ to
$C^{\infty}(\overline{\Omega})$, but with this method a
continuous solution operator between Sobolev spaces can only be
obtained with a resulting loss of regularity.  This suggests the
question whether a linear solution operator which maps $W^s_{(0,1)}(\Omega)$
to $W^s(\Omega)$ simultaneously for all $s\ge 0$ (see
the discussion in Section 5.2 in \cite{St}):
\begin{question*}
	Let $\Omega\subset\mathbb{C}^n$ be a smoothly bounded
	pseudoconvex domain.  Let $W^s_{(p,q)}(\Omega)$
	denote the Sobolev $s$ space for $(p,q)$-forms,
	where $0\le p \le n$ and $1\le q\le n$.
	Does there exist a solution operator $K$ such that
	\begin{equation*}
	K:W^s_{(p,q)} (\Omega)\rightarrow W^s_{(p,q-1)}(\Omega)
	\end{equation*}
	for all $s\ge0$, and such that  $\mdbar K f = f$
	for any \dbar-closed $f\in L^2_{(p,q)}(\Omega)$?
\end{question*}

It is this question which we study in this article in the case of
weakly pseudoconvex domains in $\mathbb{C}^2$.  We mention here
that the regularity of a solution operator is limited by the case
of a preservation of the Sobolev levels; a gain of regularity
cannot be achieved on general pseudoconvex domains.  There are
examples of convex domains with analytic discs in the boundary
which exclude the existence of a compact solution operator to
\dbar\ \cite{FuSt98} as well the compactness of certain Hankel
operators \cite{CuSa09}, and thereby exclude a solution operator
to \dbar\ which provides for a gain of regularity.

In this article, we show an operator with the mapping properties stated in the
question above can be constructed on the subspace
$W^s_{(0,1)}(\Omega)\cap \mbox{ker }\mdbar$ if a solution operator
to the $\mdbar_b$-equation can be found with analogous regularity.
We define $A_b^s(\partial\Omega)$ to be the space
\begin{equation*}
A_b^s(\partial\Omega):= \left\{ \alpha \in
W^{s}_{(0,1)}(\partial\Omega) : \int_{\partial\Omega} \alpha\wedge
\phi = 0, \forall \phi\in C^{\infty}_{(2,0)}(\partial\Omega)\cap
\mbox{ker}(\mdbar_b)\right\}
\end{equation*}
(see for instance \cite{CS}, Theorem 9.3.1), with norm given by
the Sobolev $s$ norm, $\|\cdot\|_{W^{s}(\partial\Omega)}$.
\begin{main*}
	Let $\Omega\subset\mathbb{C}^2$ be a smoothly bounded
	pseudoconvex domain.  Suppose there exists a solution operator,
	$K_b$ such that $\mdbar_b K_b g = g$ for $g\in
	A^0_b(\partial\Omega)$ and
	$K_b: A^s_b(\partial\Omega) \rightarrow
	W^{s}(\partial\Omega)$ for all 
$s\ge -1/2$.   Then
	there exists a solution operator $K$ such that $\mdbar K f = f$
	for all $f\in L^2_{(0,1)}(\Omega)\cap\mbox{ker }\mdbar$
	and
	$
	K: W^s_{(0,1)}(\Omega)\cap \mbox{ker }\mdbar\rightarrow W^{s}(\Omega)
	$ for all $s \ge 0$.
\end{main*}

The idea behind the proof is to base the construction of the
solution operator on the solution to a boundary value problem,
much as the solution to the canonical solution is based on the
\dbar-Neumann problem.  The \dbar-Neumann problem is defined as
follows.  Let $\vartheta$ denote the formal adjoint of $\mdbar$.
Let $\square=\vartheta\mdbar+\mdbar\vartheta$. The \dbar-Neumann
problem is the boundary value problem:
\begin{equation*}
\square u = f \quad \mbox{in }\Omega
\end{equation*}
with the boundary conditions
\begin{align*}
& \mdbar u \rfloor \mdbar \rho = 0,\\
&  u \rfloor \mdbar \rho = 0,
\end{align*}
where $\rho$ is a smooth defining function:
$\Omega=\{z\in\mathbb{C}^2:\rho(z)<0\}$.  Let $N$ denote the
solution operator to the \dbar-Neumann problem, i.e. as written
above, $u=Nf$.  Then $\vartheta N$ provides a solution operator to
the \dbar-equation.

As was mentioned above,
 the solution operator $\vartheta N$ does not
satisfy the conclusions of the Main Theorem.  Our approach in this
article is to relax the boundary conditions  (we eliminate the
second, Dirichlet-type, condition).

We use the technique of 
 reducing a boundary value problem
to a problem exclusively on the boundary, 
 using a Green's operator and Poisson's 
operator related to the $\square$ operator.
 The inspiration for this reduction comes from
  \cite{CNS92}.
 Several properties of the Green's operator
  and Poisson operator have been worked out
in \cite{Eh18_halfPlanes, Eh18_dno}.
 In fact, the properties
of boundary value operators stemming from
 the Poisson's operator, in particular 
regarding the Dirichlet to Neumann operator
(DNO), defined as giving the inward normal 
 derivative at the boundary
 to the solution to a Dirichlet problem,
as well as properties of 
 the Green's operator 
 motivate our particular solution.

The beginnings of this work was 
 initiated while the author was at the
University of Wuppertal and the hospitality of the University and
its Complex Analysis Working Group is sincerely appreciated.  The
author particularly thanks Jean Ruppenthal for his warm and
generous invitation to work with his group.  A visit to the
Oberwolfach Research Institute in 2013 as part of a Research in
Pairs group was also helpful in the formation of this article, for
which the author extends gratitude to the Institute as well as to
S\"{o}nmez \c{S}ahuto\u{g}lu for helpful discussions on various
mathematical topics many of which are included below, as well as
for bringing to my attention the relevant results of \cite{FuSt98}
and \cite{CuSa09}.

\section{Background information}

\label{notn}

We take a moment to fix the notation used throughout the article.
Our notation for derivatives is $\partial_t:=\frac{\partial}{\partial t}$.
We also use the index notation for derivatives: with
$\alpha=(\alpha_1,\ldots,\alpha_n)$ a multi-index
\begin{equation*}
\partial^{\alpha}_x = \partial_{x_1}^{\alpha_1}\cdots
\partial_{x_n}^{\alpha_n}.
\end{equation*}

We let $\Omega\subset\mathbb{R}^n$ and
define pseudodifferential operators on $\Omega$ as in \cite{Tr}:
\begin{defn}
	We denote by
	$\mathcal{S}^{\alpha}(\Omega)$
	the space of symbols
	$a(x,\xi)\in C^{\infty}(\Omega\times\mathbb{R}^n)$ which
	have the property that for any given compact set, $K\subset\Omega$, and for
	any $n$- tuples $k_1$ and $k_2$, there is a
	constant $c_{k_1,k_2}(K)>0$ such that
	\begin{equation*}
	\left| \partial_{\xi}^{k_1} \partial_{x}^{k_2} a(x,\xi) \right|
	\le c_{k_1,k_2}(K) \left( 1+|\xi|
	\right)^{\alpha-|k_1|}
	\qquad \forall x\in K,\ \xi\in\mathbb{R}^n.
	\end{equation*}
\end{defn}
Associated to the symbols in class $\mathcal{S}^{\alpha}(\Omega)$
are the pseudodifferential operators, denoted by
$\Psi^{\alpha}(\Omega)$ defined in
\begin{defn}
	\label{defnop}
	We say an operator $A: \mathscr{E}'(\Omega)\rightarrow
	\mathscr{D}'(\Omega)$ is in class
	$\Psi^{\alpha}(\Omega)$
	if $A$ can be written as an integral operator with symbol
	$a(x,\xi)\in \mathcal{S}^{\alpha}
	(\Omega)$:
	\begin{equation}
	\label{defnom}
	Au (x) = \frac{1}{(2\pi)^n} \int_{\mathbb{R}^n} a(x,\xi)
	\widehat{u}(\xi) e^{ix\cdot \xi} d\xi.
	\end{equation}
\end{defn}

In our use of Fourier transforms and equivalent symbols we find
cutoffs useful in order to make use of local coordinates, one of
which being a defining function for the domain.  Let $\chi(\xi)\in
C_0^{\infty}(\mathbb{R}^n)$ be such that $\chi\equiv 1$ in a
neighborhood of 0 and $\chi\equiv 0$ outside of a compact set
which includes 0.  Then we reserve the notation $\chi'$ to denote
functions which are 0 near the origin: $\chi'(\xi):=1-\chi(\xi)$
where $\chi(\xi)$ is as defined above.

We use $\widetilde{\ \ }$ to indicate transforms in tangential
directions.  Let $p\in\partial\Omega$ and let
$(x_1,\ldots,x_{n-1},\rho)$ be local coordinates around $p$,
$(\rho<0)$.  Let $\chi_p(x,\rho)$ denote a cutoff which is $\equiv
1$ near $p$ and vanishes outside a small neighborhood of $p$ on
which the local coordinates $(x,\rho)$ are valid.  Then with $u\in
L^2(\Omega)$ we write
\begin{align*}
& \widehat{\chi_p u}(\xi,\eta)=
\int \chi_p u(x,\rho) e^{-ix\xi} e^{-i\rho\eta} dx d\rho\\
&\widetilde{\chi_p u}(\xi,\rho)=
\int \chi_p u(x,\rho) e^{-ix\xi}  dx.
\end{align*}
We also use the $\widetilde{\ \ }$ notation when describing
transforms of functions supported on the boundary.  With notation
and coordinates as above, we let $u_b(x)\in L^2(\partial\Omega)$
and write
\begin{equation*}
\widetilde{\chi_p(x,0) u_b}(\xi)=
\int \chi_p(x,0) u_b(x) e^{-ix\xi}  dx.
\end{equation*}

If we let
$\chi_j$ be such that $\{\chi_j\equiv 1\}_j$ is a covering of
$\Omega$, and let $\varphi_j$ be a partition of unity subordinate
to this covering, then locally, we describe 
an operator
$A: \mathscr{E}'(\Omega)\rightarrow
\mathscr{D}'(\Omega)$ in terms of its symbol,
$a(x,\xi)$ according to
\begin{equation*}
Au = \frac{1}{(2\pi)^n} \int a(x,\xi) \widehat{\chi_j
	u}(\xi)
d\xi
\end{equation*}
on $\mbox{supp }\varphi_j$.  Then we can describe the operator $A$
globally on all of $\Omega$ by
\begin{equation}
\label{alld}
Au = \frac{1}{(2\pi)^n} \sum_j \varphi_j \int a(x,\xi) \widehat{\chi_j
	u}(\xi)
d\xi.
\end{equation}
The difference arising between the definitions in \eqref{defnom}
and \eqref{alld} is a smoothing term \cite{Tr}, which we write as
$\Psi^{-\infty}u$, to use the notation of Definition
\ref{defnop}.

Pseudodifferential operators on the 
 boundary, that is, in the class
$\Psi^k(\partial\Omega)$ for some $k$,
will be marked with a subscript "b".
 Thus if $\phi_b$ is a distribution
  with support on $\partial\Omega$,
$\Psi^{-1}_b \phi_b $ denotes an 
 operator in class $\Psi^{-1}(\partial\Omega)$
  applied to $\phi_b$.

We follow \cite{CNS92} in setting up our boundary value problem
(which is similar to the setup of the \dbar-Neumann problem).
 We let
$\rho\in C^{\infty}$ be a defining function for
$\Omega$: $\Omega=\{z\in\mathbb{C}^2:
\rho(z)< 0\}$,
$\Omega$, a smoothly bounded pseudoconvex
domain.
We let $U$ be an open neighborhood of $\partial\Omega$ such that
\begin{align*}
&\Omega \cap U = \{ z\in U| \rho(z)<0 \};\\
& \nabla \rho(z) \neq 0 \qquad \mbox{for } z\in U.
\end{align*}

We define an orthonormal 
 (with respect to the Euclidean
  metric) 
 frame of $(1,0)$-forms on a neighborhood
$U$ with $\omega_1,\omega_2$ where
$\omega_2=\sqrt{2}\partial\rho$, and $L_1, L_2$ the dual frame.
We thus can write
\begin{align}
\nonumber
L_1&=\frac{1}{2}(X_1-iX_2) + O(\rho)\\
\label{l2}
L_2&= \frac{1}{\sqrt{2}} \frac{\partial}{\partial\rho}
+iT+O(\rho)
\end{align}
where $\partial/\partial\rho$ is the vector field dual to
$d\rho$, and $X_1$, $X_2$, and $T$ are tangential fields.
The special tangential operator
$T=\frac{1}{2i} (L_2 - \Lb_2)$
will receive extra attention in this
paper.  
We also use the notation
$L_{bj}$ to denote 
$L_j$ restricted to $\rho=0$,
and $T^0 = T\big|_{\rho=0}$.
We can expand the vector fields $L_1$ and $T$ as in
\cite{CNS92} as
\begin{align*}
L_1 =& L_1^0 + \rho L_1^1 + \cdots\\
L_2 = & \frac{1}{\sqrt{2}} \frac{\partial}{\partial\rho} + i\left(
T^0 + \rho T^1 + \cdots\right).
\end{align*}
We then choose coordinates $(x_1,x_2,x_3)$ on $\partial\Omega$ near a
point $p\in\partial\Omega$, in
terms of which
the vector fields $L_1^0$ and $T^0$ are given by
\begin{align*}
T^0=&\frac{\partial}{\partial x_3}\\
L_1^0=& \frac{1}{2} \left( \frac{\partial}{\partial x_1} -i
\frac{\partial}{\partial x_2} \right) + O(x-p).
\end{align*}
To emphasize the 0 superscript refers to restriction to 
 the boundary, we will use the notation
\begin{align*}
& L_{b1}:= L_1^0 = L_1\big|_{\rho=0}\\
& \Lb_{b1}:= \Lb_1^0 = \Lb_1\big|_{\rho=0}.
\end{align*}
We also use the notation, $R$, to 
 denote the restriction to the boundary
operator.  Thus,
\begin{equation*}
 R\circ \Lb_1 \equiv \Lb_{b1}.
\end{equation*}

 We define the scalar function $s$ by
\begin{equation*}
\mdbar \bar{\omega}_1
= s \bar{\omega}_1 \wedge \bar{\omega}_2.
\end{equation*}
With respect to the coordinates, $z_1$ and $z_2$
\begin{align*}
\bar{\omega}_1=& \sqrt{2} \left( \frac{\partial\rho}{\partial z_2} d\bar{z}_1
- \frac{\partial\rho}{\partial z_1} d\bar{z}_2\right),\\
\mdbar \bar{\omega}_1
=& - \sqrt{2}\left(\frac{\partial^2\rho}{\partial \bar{z}_1 \partial z_1}
+\frac{\partial^2\rho}{\partial \bar{z}_2 \partial z_2} \right)
d\bar{z}_1\wedge d\bar{z}_2\\
=& - 2 \sqrt{2}\left(\frac{\partial^2\rho}{\partial \bar{z}_1
	\partial z_1} +\frac{\partial^2\rho}{\partial \bar{z}_2 \partial
	z_2} \right) \bar{\omega}_1\wedge \bar{\omega}_2,
\end{align*}
and so
\begin{equation*}
s(z_1,z_2)
= - 2\sqrt{2} \left(\frac{\partial^2\rho}{\partial \bar{z}_1 \partial z_1}
+\frac{\partial^2\rho}{\partial \bar{z}_2 \partial z_2} \right).
\end{equation*}

If we write a $(0,1)$-form, $u$, as
\begin{equation*}
 u = u_1\omegab_1 + u_2\omegab_2,
\end{equation*}
and its boundary values, $u_b:=Ru$ as
\begin{equation*}
u_b = u_b^1\omegab_1 + u_b^2\omegab_2,
\end{equation*}
the boundary condition
 $\mdbar u \rfloor \mdbar \rho = 0$ in
the \dbar-Neumann problem can be expressed
as
\begin{equation*}
 \Lb_2 u_b^1 - s_0u_b^1 -\Lb_1 u_b^2 =0,
\end{equation*}
where $s_0:= R s$.

On the boundary of our domain in $\mathbb{R}^4$,
we further break up the transforms with the use of the following
microlocal decomposition into three regions, as in for
instance
\cite{Ch91, KoNi06, K85, Ni06}.  We write $\xi_{1,2}:=(\xi_1,
\xi_2)$, and
define the three regions
\begin{align*}
& \lrc^+ = \left\{\xi \big| \xi_3 \ge \frac{1}{2}|\xi_{1,2}|,\ |\xi|\ge 1 \right\}\\
& \lrc^0 = \left\{\xi \big| -\frac{3}{4}|\xi_{1,2}| \le \xi_3 \le \frac{3}{4} |\xi_{1,2}|\right\}
\cup    \{\xi \big| |\xi|\le 1\}\\
& \lrc^- = \left\{\xi \big| \xi_3 \le -\frac{1}{2}|\xi_{1,2}|,\ |\xi|\ge 1\right\}.
\end{align*}
Associated to the three regions we define the functions
$\psi^+(\xi)$, $\psi^0(\xi)$, and $\psi^-(\xi)$ with the
following properties: $\psi^+,\psi^0,\psi^-
\in C^{\infty}$, are symbols of order 0
with values in $[0,1]$,
$\psi^+$ (resp. $\psi^0$, resp. $\psi^-$) restricted to
$|\xi|=1$ has compact support in $\lrc^+\cap \{|\xi|=1\}$ (resp.
$\lrc^0\cap \{|\xi|=1\}$, resp. $\lrc^-\cap \{|\xi|=1\}$)
with $\psi^-(\xi)=\psi^+(-\xi)$
and $\psi^0$ is given by $\psi^0(\xi)= 1
-\psi^+(\xi) - \psi^-(\xi)$.
Furthermore,
 for $|\xi|\ge r$ for some $r<1$,
  $\psi^-(\xi)=
  \psi^-\left(\frac{\xi}{|\xi|}\right)$
(resp.
$\psi^0(\xi)= \psi^0\left(\frac{\xi}{|\xi|}\right)$,
$\psi^+(\xi)=\psi^+\left(\frac{\xi}{|\xi|}\right)$).
 $\psi^0(\xi) \equiv 1$ in a neighborhood 
  of the origin,
and the relation
 $\psi^0(\xi)+\psi^+(\xi) + \psi^-(\xi)= 1$
  is to hold on all of $\mathbb{R}^3$.

The support of $\psi^0$ is contained in $\lrc^0$, and from the
above requirements we have the support of $\psi^+$ (resp.
$\psi^-$) is contained in $\lrc^+\cup \{|\xi| \le 1\}$ (resp. $\lrc^- \cup \{c\le |\xi| \le
1\}$).   We make the further restrictions that the supports of
$\psi^+$ and $\psi^-$ are contained in conic neighborhoods; we
define
\begin{align*}
& \widetilde{\lrc}^+ = \left\{\xi \big| \xi_3 \ge \frac{1}{2}|\xi_{1,2}| \right\}\\
& \widetilde{\lrc}^- = \left\{\xi \big| \xi_3 \le -\frac{1}{2}|\xi_{1,2}|\right\}.
\end{align*}
We also assume that the support of $\psi^+$ and $\psi^-$ are
contained in $\widetilde{\lrc}^+$ and $\widetilde{\lrc}^-$,
respectively, such that the restrictions,
$\psi^+\big|_{\{|\xi|\le 1\}}$ and $\psi^-\big|_{\{|\xi|\le 1\}}$
have support which is relatively compact in the interior of the
regions $\widetilde{\lrc}^+$ and $\widetilde{\lrc}^-$,
respectively.

We note that due to the radial extensions from the unit sphere,
the functions $\psi^-(\xi)$, $\psi^0(\xi)$, and $\psi^+(\xi)$
are symbols of zero order pseudodifferential operators.
The operator $\Psi^+$ (resp. $\Psi^-$)
is defined as the operator with symbol $\psi^+$ (resp.
$\psi^-$).  We do not have need for the
operator defined by the symbol $\psi^0$ and as the above notation
would conflict with our notations of generic pseudodifferential
operators of order 0, we have left out this operator.

We will make the
assumption that $\psi^-\equiv 1$ in a neighborhood of
$\lrc^-\cap (\lrc^0)^c$.  This is to ensure that operators
formed by commutators with $\Psi^-$ have symbols whose restrictions to
the sphere $|\xi|=1$ have compact support in the
region $\lrc^-\cap \lrc^0\cap \{|\xi|=1\}$.

Similarly, we define $\psi^-_D(\xi)\in C^{\infty}(\widetilde{\lrc}^-)$ with the property
$\psi^-_D(\xi)= \psi^-_D(\xi/|\xi|)$ for $|\xi|\ge 1$,
and such that
$\psi^-_D \equiv 1$ on $\mbox{supp } \psi^-$.  And, as with $\psi^-$, the restriction to
the disc, $\psi^-_D\big|_{\{|\xi|\le 1\}}$, has relatively compact support
in the interior of $\widetilde{\lrc}^-$.
We shall have the occasion to use the operator defined by the
symbol $\psi^-_D(\xi)$.  This operator we denote $\Psi^-_D$.  In
the terminology of \cite{K85} we say $\Psi^-_D$ {\it dominates}
$\Psi^-$.

We further use the notation
$ u^{\psi^-}$
as a short-hand for $\Psi^- u$, with similar meanings for
$u^{\psi^0}$ and $u^{\psi^+}$.
The use of $ u^{\psi^-}$ (resp. $ u^{\psi^0}$, $u^{\psi^+}$)
thus has the advantage of allowing us to consider the symbol of
a boundary operator in only one microlocal region, $\widetilde{\lrc}^-$ (resp.
$\lrc^0$, resp. $\widetilde{\lrc}^+$); naturally it holds that
$u=u^{\psi^-}+u^{\psi^0}+u^{\psi^+}$, modulo smooth terms.
We shall use such microlocalizations in Section \ref{secTheBvp} to
obtain a solution to the boundary value problem as a sum of three
terms, each solving an equation relating to symbols whose
transform variables are restricted to one of $\widetilde{\lrc}^- $, $\lrc^0$,
or $\widetilde{\lrc}^+$.

\section{A modified \dbar-Neumann type boundary value problem}

The \dbar-Neumann problem is the vector-valued 
 (with forms written as vectors)
  boundary-value
problem:
\begin{equation*}
\square u = f \qquad \mbox{in}\ \ \Omega,
\end{equation*}
where $\square=\vartheta\mdbar
+ \mdbar\vartheta$,
with the boundary conditions
\begin{align*}
\nonumber
\overline{L}_2 u_1 -su_1=&0 \\
u_2=&0
\end{align*}
on $\partial\Omega$.  In our modified problem, we elimiate the
condition $u_2=0$ on the
boundary; this leads to the consideration of forms $u$ which are no longer
in the domain of $\mdbar^{\ast}$ and it is for this reason we
describe
the operator $\square$ in terms of the formal adjoint, rather than
with $\mdbar^{\ast}$ as is common in the theory of the
\dbar-Neumann problem (note that on  $\mbox{dom}(\mdbar^{\ast})$,
we have $\mdbar^{\ast}=\vartheta$).  We now describe the modified problem.

We
consider
\begin{equation*}
\square u = f \qquad \mbox{in }\Omega,
\end{equation*}
with the boundary conditions
\begin{align}
\label{bvdar}
\overline{L}_2 u_1 -s_0 u_1-\Lb_1u_2=&0.
\end{align}
With the help of Green's operator and Poisson's 
 operator we can reduce the boundary value problem to the
boundary (see also \cite{CNS92, Eh18_dno}).

We denote by $P$ a Poisson's operator for the boundary value
problem
\begin{align*}
&2\square \circ P  = 0 \qquad \mbox{on } \Omega\\
& R\circ P = I \qquad 
\mbox{on } \partial\Omega.
\end{align*}
The operators $P_1$ and
$P_2$ denote respectively the first and second components of
the solution given by the operator $P$:
\begin{equation*}
P(u_b) = P_1(u_b)\omegab_1 + P_2(u_b)\omegab_2.
\end{equation*}
The DNO, given by the derivative of the Poission's operator
restricted to the boundary,
\begin{equation*}
N^-u_b= R\circ \frac{\partial}{\partial\rho}  P(u_b),
\end{equation*}
where $R$ denotes the operation of restriction to the boundary,
is thus a matrix of operators.  We concentrate on the first
component and write
\begin{equation}
 \label{DNOfirstComponent}
R\circ \frac{\partial}{\partial\rho}  P_1(u_b)
= N^-_1u_b^1 + N^-_2u_b^2,
\end{equation}
where $N^-_1$ is the $(1,1)$ entry of the DNO matrix operator
and $N_2^-$ the $(1,2)$ entry.
 
We define
 the symbol of class $\lrs^1(\partial\Omega)$
\begin{equation*}
 |\Xi(x,\xi)| = \sqrt{-2\sigma(T^0)^2 -
  2\sigma(L_1^0)\sigma(\Lb_1^0)}
  \end{equation*}
 and the corresponding operator, $|D|$, by
\begin{equation*}
 \sigma(|D|) = |\Xi(x\xi)|.
\end{equation*}
 From Theorem 4.4 \cite{Eh18_dno}
\begin{thrm}
	\label{thrmdno1}
	\begin{equation}
	\label{dnoprincipal}
	N^- g = |D| g_b+
	\Psi^0_b g_b + R^{-\infty}_b,
	\end{equation}
with
\begin{equation*}
 \| R_b^{-\infty} \|_{W^s(\partial\Omega)}
  \lesssim \|g\|_{L^2(\partial\Omega)}
\end{equation*}
 for all $s\ge 0$.
\end{thrm}
 The first term on the right-hand side is understood to mean 
a diagonal operator with diagonal entries given by the
 operator $|D|$.
In particular, 
\begin{equation*}
 \sigma(N_1^-) = |\Xi(x,\xi)|.
\end{equation*}

We have the following well-known 
 estimates for the Poisson operator (see also
  Theorem 4.3 \cite{Eh18_dno}):
\begin{thrm}
	\label{poissest}
	For $s\ge 0$
	\begin{equation*}
	\| P (g) \|_{W^{s+1/2}(\Omega)} \lesssim
	\|g_b\|_{W^{s}(\partial \Omega)}.
	\end{equation*}
\end{thrm}
Furthermore, the principal 
term of the Poisson operator is
calculated in \cite{Eh18_dno}. 
We define $\Theta^+$ to be the 
operator with symbol
\begin{equation*}
\sigma(\Theta^+) =    \frac{i}
{\eta + i|\Xi(x,\xi)|}.
\end{equation*} 
Then we can write
\begin{equation}
\label{formPoisson}
Pg = \Theta^+g + \Psi^{-2}g + R^{-\infty}
\end{equation}
where $R^{-\infty}$ denotes
smooth terms which can be estimated
by
\begin{equation*}
\| R^{-\infty} \|_{W^s(\Omega)}
\lesssim \| g\|_{L^2(\partial \Omega)}
\end{equation*}
for all $s\ge 0$ (see Theorem 4.1 in
\cite{Eh18_dno}).

 We define the Green's operator corresponding to
$2\square$ as a solution operator,
$G$ mapping $(0,1)$-forms on $ \Omega$
to $(0,1)$-forms on $\Omega$, to
\begin{align*}
&2\square\circ G =I\\
&R\circ G = 0.
\end{align*}
  If $f=f_1 \bar{\omega}_1 + f_2
\bar{\omega}_2$, we write
\begin{equation*}
G(f)=G_1(f) \bar{\omega}_1 + G_2(f) \bar{\omega}_2,
\end{equation*}
where
\begin{align*}
G_1(f) & = G_{11}(f_1) + G_{12}(f_2)\\
G_2(f) & = G_{21}(f_1) + G_{22}(f_2).
\end{align*}

From Theorem 3.2 
 in \cite{Eh18_halfPlanes}
\begin{thrm}
	\label{greenest}
	Let $G(f)$ denote the solution, $u$, to 
	the boundary value problem 
$\square u=f$ with the boundary condition
 $u=0$ on $\partial\Omega$.  Then
	\begin{equation*}
	\| G (f) \|_{W^{s+2}(\Omega)} \lesssim
	\|f\|_{W^{s}(\Omega)},
	\end{equation*}
	for $s\ge 0$.
\end{thrm}
And from Theorem 3.3 in \cite{Eh18_halfPlanes},
 \begin{thrm}
	\label{greender}
	Let $\Theta^{-}\in \Psi^{-1}(\Omega)$ be the operator with symbol
	\begin{equation*}
	\sigma(\Theta^{-}) = \frac{i}{\eta-i|\Xi(x,\xi)|}.
	\end{equation*}
	Then
	\begin{equation}
	\label{eqngreender}
	R \circ \frac{\partial}{\partial \rho}
	\circ G (g) =  R\circ \Theta^{-} g
	+\Psi_b^{-1}\circ
	 R\circ \Psi^{-1}g + 
+ R\circ \Psi^{-2}g +R_b^{-\infty},
	\end{equation}
where $R_b^{-\infty}$ denote smooth
terms which can be
estimated by
\begin{equation*}
 \| R_b^{-\infty}\|_{W^s(\partial\Omega)}
  \lesssim \|g\|_{L^2(\Omega)}.
\end{equation*}
\end{thrm}

 We now follow \cite{CNS92} to
reduce to the boundary.  Recall the 
boundary condition:
\begin{equation}
\label{bv}
\overline{L}_2 u_1 -s_0 u_1-\Lb_1u_2=0.
\end{equation}
There are possibly many solutions to the boundary value problem
(note that as stated we leave the Dirichlet type condition open in
contrast to the \dbar-Neumann problem), and we will isolate one
particular approximate solution.

With the solution $u$ written
$u= u_1 \bar{\omega}_1 + u_2 \bar{\omega}_2$, recall we write its restriction to
$\partial\Omega$ as
\begin{equation*}
u_b = u_b^1 \bar{\omega}_1 + u_b^2 \bar{\omega}_2.
\end{equation*}

We consider Equation \ref{bv} microlocally and look for solutions
\begin{equation*}
u_b= u_b^- + u_b^0 + u_b^+
\end{equation*}
where $u_b^-$ can be written
\begin{equation*}
u_b^- = u_b^{1,-}\omegab_1 + u_b^{2,-}\omegab_2
\end{equation*}
and
$u_b^{j,-}$
are described in terms of
pseudodifferential operators whose symbols have support in
$\widetilde{\lrc}^-$ (later these operators will be seen to have the form
of compositions of the operators $\Psi^-$ or $\Psi_D^-$
with operators acting on the data form, $f$; we recall the convention that $\psi^-_D(\xi)
\equiv 1$ on $\mbox{supp }\psi^-$).  We of course have similar
meanings for $u^0$ and $u^+$.

A solution to
 $\square u =f$,
  under condition \eqref{bv} is given
by
\begin{equation}
\label{solnpg}
u=G(2f)+P(u_b),
\end{equation}
We write the boundary condition in terms of the first component of the
DNO as in \eqref{DNOfirstComponent}.

Then locally we can write condition \eqref{bv} as
\begin{align*}
0=&R\circ \left(\frac{1}{\sqrt{2}} \frac{\partial}{\partial\rho}
-iT^0\right)  u^1_b -s_0 u_b^1-\Lb_1u_b^2\\
=&R\circ \left(\frac{1}{\sqrt{2}} \frac{\partial}{\partial\rho}
-iT^0\right) \big(G_1(2f)
+ P_1(u_b) \big) -s_0 u_b^1 -\Lb_1u_b^2\\
=&  \frac{1}{\sqrt{2}} R\circ \Theta^{-} (2f_1)
  +
\left(\frac{1}{\sqrt{2}} N^-_1
-iT^0\right)u_b^1 +\Psi^0_b u_b^1
 -\Lb_1u_b^2
+ \frac{1}{\sqrt{2}} N^-_2u_b^2
,
\end{align*}
modulo $\Psi_b^{-1}\circ
R\circ \Psi^{-1} f$, 
$R\circ \Psi^{-2}f$, and 
  smooth terms $R_b^{-\infty}$,
using Theorem \ref{greender} in the last line.  We rewrite this
as
\begin{equation}
\label{bndrycndn}
\left(\frac{1}{\sqrt{2}} N^-_1
-iT^0\right)u_b^1 +
\Psi^0_b u_b^1 -\Lb_1u_b^2
+ \frac{1}{\sqrt{2}} N^-_2u_b^2 = -\frac{2}{\sqrt{2}}
R\circ \Theta^{-} f_1,
\end{equation}
modulo  
$\Psi^{-1}_b u_b^2$,
$\Psi_b^{-1}\circ
R\circ \Psi^{-1} f$, 
$R\circ \Psi^{-2}f$ and  
 the smooth terms $R_b^{-\infty}$.

Our approximate solution $u$, will be determined via
\eqref{solnpg}
by its boundary values.

\section{Relations among some boundary value operators}

We first examine the $N_2^-$ operator in 
 \eqref{bndrycndn} above.  From
  \cite{Eh18_dno},
$N_2^-$ can be written in the form
\begin{equation*}
 \frac{1}{2} \left(N_1^{-}\right)^{-1} \circ A_{12}
\end{equation*}
modulo lower order terms
(see the non-diagonal terms in Theorem 4.6 in \cite{Eh18_dno}),
where $A$ refers to the first order tangential 
 operator in $2\square$, restricted to
  $\partial\Omega$, and $A_{12}$, the operator in 
the $(1,2)$-entry.  From the discussion preceding
 Proposition 3.1 of \cite{Eh18_dno} 
 (see also (2.22) of \cite{CNS92}),
  we have
\begin{equation*}
 A_{12} =2 R\circ [L_2,\Lb_1] \qquad \mbox{{mod }} \Lb_{b1}.
\end{equation*}

Without loss of generality we assume the Levi matrix is diagonal,
so that immediately we have
\begin{equation*}
 \left<R\circ [L_2,\Lb_1],T^0\right> = 0,
\end{equation*}
where $\left<\cdot, \cdot\right>$ denotes the inner product of 
 vector fields. 
We also have
\begin{lemma}
	\label{lemL1Dot}
\begin{equation*}
 \left<[L_2,\Lb_1], L_1 \right>  
  =2i\left<[T,\Lb_1], L_1 \right>.
\end{equation*}
\end{lemma}
\begin{proof}
 From 
\begin{equation*}
 \left<[\Lb_2,\Lb_1], L_1 \right>  =0,
\end{equation*}
we have
\begin{align*}
  \left<[L_2,\Lb_1], L_1 \right>  =&
    \left<[\Lb_2,\Lb_1], L_1 \right> 
     +2i\left<[T,\Lb_1], L_1 \right>\\
     =& 2i\left<[T,\Lb_1], L_1 \right>.
\end{align*}
\end{proof}

\begin{lemma}
\label{lemNT}
\begin{equation*}
\Psi^-_D\circ 
 \left(N_1^-\right)^{-1}  \circ T^0 = -\frac{i}{\sqrt{2}}
 \Psi^-_D
 + \Psi^{-1}_b \circ \Lb_{b1}
\end{equation*}
modulo $\Psi^{-1}(\partial\Omega)$.
\end{lemma}
\begin{proof}
 Define 
\begin{equation*}
 \kappa := \frac{\sigma(\Lb_{b1}) \sigma(L_{b1}) } {\xi_3^2}.
\end{equation*}
Recall the symbol
 $\psi^-_D$ with support in the region
  $\xi_3\le -\frac{1}{2} |\xi_{1,2}|$.  
We take
 $U$ to be a small enough conic neighborhood
  of $(0,\mbox{supp } \psi^-_D)$.  
  In the conic neighborhood
$U$,
$\kappa<c$ for some $c<1$ and
 we have
\begin{align}
 \nonumber
|\Xi(x,\xi)| =&  \sqrt{2}|\xi_3| \sqrt{1+\kappa} \\
 \nonumber
=&  \sqrt{2} |\xi_3| \left(1+ \frac{1}{2}\kappa - \frac{1}{8}\kappa^2 +
\cdots\right)\\
\nonumber
=&\sqrt{2} |\xi_3|+ \sqrt{2} |\xi_3| \left(\frac{1}{2}\kappa - \frac{1}{8}\kappa^2 +
\cdots\right) \\
\label{XiExpressTLL}
=& \sqrt{2}\sigma(iT^0)+ 
\sqrt{2} \sigma(\Lb_{b1})
 \frac{\sigma(L_{b1})}{|\xi_3|}\left(\frac{1}{2} 
  - \frac{1}{8}\kappa + \cdots\right).
\end{align}
Since in the neighborhood, $U$,
 the infinite sum in parentheses
converges uniformly, and as 
$\psi^-_D \kappa
  \in
\mathcal{S}^0(\partial\Omega)$, we see that by differentiating the
power series the symbol given by
\begin{equation}
\label{bo-}
\sigma(B_{0}) = \psi^-_D(\xi) \frac{\sigma(L_{b1})}{|\xi_3|}\left(\frac{1}{2} - \frac{1}{8}\kappa +
\cdots\right)
\end{equation}
defines an operator $B_0\in
\Psi^0(\partial\Omega)$.

Dividing \eqref{XiExpressTLL} by
 $|\Xi(x,\xi)|$ 
 and reverting to operators yields
to highest order, i.e. modulo $\Psi^{-1}(\partial\Omega)$,
\begin{equation*}
( N_1^{-})^{-1}\circ T^0 = -\frac{i}{\sqrt{2}}
  + \Psi^{-1}_b \circ \Lb_{b1}
\end{equation*}
in the microlocal neighborhood defined
 by the support of $\psi^-_D$.
\end{proof}

\begin{lemma} 
 \label{lemmaThetaExpn}
Let $\Theta^-$ be defined as in 
 Theorem \ref{greender}.  Then
\begin{align*}
\Psi^-\circ \Theta^-  
  =& \frac{3}{4} \Psi^-\circ (N_1^-)^{-1} \circ
   R 
  -\frac{1}{\sqrt{2}} \Psi^-\circ (N_1^-)^{-1} \circ
  R \circ \Theta^{-} \circ  \Lb_2 \\
& + \Psi^-_D\circ \Psi_b^{-1}\circ
 \Lb_{b1}\circ\Psi^{-1}.
\end{align*}
\end{lemma}
\begin{proof}
\begin{align}
\nonumber
R\circ \Theta^{-} \circ  \Lb_2 \phi
=& \frac{i}{(2\pi)^2}
\int\frac{1}{\eta-i|\Xi(x,\xi)|}
\widehat{\Lb_2 \phi} e^{ix\xi} d\eta d\xi\\
\nonumber
=&\frac{i}{(2\pi)^2} \frac{1}{\sqrt{2}}
\int\frac{1}{\eta-i|\Xi(x,\xi)|}
\widetilde{\phi}(\xi,0) e^{ix\xi} d\eta d\xi\\
\nonumber
&-  \frac{1}{(2\pi)^2}
\int\frac{\frac{1}{\sqrt{2}}\eta-i\xi_3}
{\eta-i|\Xi(x,\xi)|}
\widehat{ \phi}(\xi,\eta) e^{ix\xi} d\eta d\xi\\
\label{RThetaL2}
=& -
\frac{3}{2\sqrt{2}}
R \phi - \left(
\frac{1}{\sqrt{2}} N_1^- +iT^0 
\right)\circ \Theta^- \phi.
\end{align} 
Now using Lemma \ref{lemNT} 
 for the last term, we can write
\begin{equation*}
 \Psi^- \circ \left(
 \frac{1}{\sqrt{2}} N_1^- +iT^0 
 \right)\circ \Theta^- \phi
   = \sqrt{2} \Psi^- \circ N_1^- \circ
    \Theta^- \phi + \Psi^-_D\circ 
  \Lb_{b1}\circ\Psi^{-1}\phi
\end{equation*}
Inserting this expression into 
 \eqref{RThetaL2} and rearranging yields the Lemma.
\end{proof}

\begin{lemma} 
\label{lemNLTL}
Modulo $\Psi^{-2}(\partial\Omega)$,
 \begin{equation*}
\Psi^-_D\circ
  \left[N_1^{-1},\Lb_{b1}\right]  
 = -i \sqrt{2}\Psi^-_D\circ
  (N_1^-)^{-2} \circ
 [T^0,\Lb_{b1}] + \Psi^{-2}_b \circ \Lb_{b1}
 \end{equation*}
\end{lemma}
\begin{proof}
 Using a symbol expansion, we see
\begin{align}
\nonumber
 \sigma_{-1}\left( \left[N_1^{-1},\Lb_{b1}\right]  \right)
  =& -i \Big( \partial_{\xi} \left(\Xi^2(x,\xi) \right)^{-\frac{1}{2}}
    \cdot \partial_x \sigma \left(
        \Lb_{b1}
     \right)\\
 \nonumber
     &\qquad -
   \partial_{x} \left(\Xi^2(x,\xi) \right)^{-\frac{1}{2}}
   \cdot \partial_{\xi} \sigma \left(
   \Lb_{b1}
   \right)  
 \Big) \\
\nonumber
=& \frac{i}{2|\Xi(x,\xi)|^3}
 \left(
 	 \partial_{\xi}  \Xi^2(x,\xi) 
 	\cdot \partial_x \sigma \left(
 	\Lb_{b1}
 	\right) -
 	\partial_{x}  \Xi^2(x,\xi) 
 	\cdot \partial_{\xi} \sigma \left(
 	\Lb_{b1}
 	\right)  
 \right)\\
\label{symbolN1Lb}
 =& -\frac{1}{2} \sigma_{-1} \left(
    \left(N_1^-\right)^{-3} \circ
      \left[ 
        (N_1^-)^2, \Lb_{b1}
      \right]
 \right).
\end{align}

Furthermore, since 
\begin{equation*}
  (N_1^-)^2 = -2(T^0)^2 - 2L_{b1}\Lb_{b1} 
\end{equation*}
modulo lower order terms,
we have
\begin{equation*}
 \left[ 
(N_1^-)^2, \Lb_{b1}
\right] =  - 4T^0\circ[T^0,\Lb_{b1}]
 + \Psi^1_b \circ \Lb_{b1}
\end{equation*}
modulo $\Psi^1(\partial\Omega)$.
Inserting this relation into 
 \eqref{symbolN1Lb} yields
\begin{equation*}
 \left[N_1^{-1},\Lb_{b1}\right] 
= 2(N_1^-)^{-2} \circ
 \left((N_1^-)^{-1} \circ T^0\right)  \circ
[T^0,\Lb_{b1}] + \Psi^{-2}_b \circ \Lb_{b1}
\end{equation*}
modulo lower order terms.
Using Lemma \ref{lemNT}, we can 
 replace the $(N_1^-)^{-1} \circ T^0$ 
term with
$-i/\sqrt{2}$, and we have
\begin{equation*}
\Psi^-_D\circ
\left[N_1^{-1},\Lb_{b1}\right] 
 = -i \sqrt{2}
 \Psi^-_D\circ
  (N_1^-)^{-2} \circ
 [T^0,\Lb_{b1}] + \Psi^{-2}_b \circ \Lb_{b1}
\end{equation*}
modulo lower order terms, which was to be
 proved.
\end{proof}

Combining Lemmas
 \ref{lemL1Dot} and \ref{lemNT} we
see that the operator
 $N_2^-$ is essentially equivalent
to the commutation operator
 $\left[N_1^{-1},\Lb_{b1}\right]$
composed with the absolute boundary
 derivative, $|D|$.
  We illustrate this in the following
 proposition:
\begin{prop}
	\label{propN2N1}
Modulo $\Psi^{-2}(\partial\Omega)$,
 \begin{equation*}
  \frac{1}{\sqrt{2}} 
  \Psi^-_D\circ N_2^- \circ (N_1^-)^{-1}
     =- \Psi^-_D\circ
       \left[N_1^{-1},\Lb_{b1}\right]  
     +\Psi^{-2}_b \circ \Lb_{b1}.
 \end{equation*}
\end{prop}
\begin{proof}
	From Lemma \ref{lemL1Dot} we have
\begin{equation*}
 [L_2,\Lb_1] = 2i [T,\Lb_1] + \Psi^0_b \circ 
  \Lb_{b1}.
\end{equation*}
Hence, with Lemma \ref{lemNLTL},
 we have, modulo 
 $\Psi^{-2}(\partial\Omega)$,
\begin{align*}
\frac{1}{\sqrt{2}}
\Psi^-_D\circ
 N_2^- \circ (N_1^-)^{-1}=&
  \frac{1}{2\sqrt{2}}
 \Psi^-_D\circ
   \left(N_1^{-}\right)^{-2} \circ A_{12}\\
 =&i\sqrt{2}
\Psi^-_D\circ
 \left(N_1^{-}\right)^{-2} \circ [T^0,\Lb_1]
+\Psi^{-2}_b \circ \Lb_{b1}\\
 =& -  \Psi^-_D\circ
  \left[N_1^{-1},\Lb_{b1}\right]  
 +\Psi^{-2}_b \circ \Lb_{b1}.
\end{align*}
\end{proof}

\section{The boudary solution with estimates}
\label{secTheBvp}

We return to \eqref{bndrycndn} 
and first look for solutions $u_b^{1,-}$ and
 $u_b^{2,-}$
for the equation corresponding to the region $\widetilde{\lrc}^-$:
\begin{align}
\nonumber
\left(\frac{1}{\sqrt{2}} N^-_1
-iT^0\right)u_b^{1,-} +&\Psi^0_b u_b^{1,-}
\\
\label{bndrycndn-}
& -\Lb_1u_b^{2,-}
+ \frac{1}{\sqrt{2}} N^-_2u_b^{2,-} =
-\frac{2}{\sqrt{2}}
\left(R \circ \Theta^{-} f_1\right)^{\psi^-}
,
\end{align}
modulo error terms.  We recall the notation from Section
\ref{notn} in which we write $\left(R \circ \Theta^{-} f_1\right)^{\psi^-}
=\Psi^- \circ 
 \left(R \circ \Theta^{-} f_1\right)$.

We first use Lemma \ref{lemmaThetaExpn}
 for the
 term $-\sqrt{2} R\circ\Theta^- f_1$, 
using the hypothesis that $f$ is \dbar-closed;
 for \dbar-closed $f$, we have the relation
 \begin{equation*}
 (\Lb_2 -s) f_1 - \Lb_1 f_2=0.
 \end{equation*}
We have
\begin{align*}
 -\frac{2}{\sqrt{2}}&
 \left(R \circ \Theta^{-} f_1\right)^{\psi^-}\\
   & \qquad = - \frac{3}{2\sqrt{2}}
     \Psi^-\circ (N_1^-)^{-1} \circ
   R f_1
   +\Psi^-\circ (N_1^-)^{-1} \circ
   R \circ \Theta^{-} \circ  \Lb_2 f_1\\
   & \qquad \quad + \Psi^-_D\circ \Psi_b^{-1}\circ
   \Lb_{b1}\circ R\circ\Psi^{-1}f_1\\
& \qquad = - \frac{3}{2\sqrt{2}}
\Psi^-\circ (N_1^-)^{-1} \circ
R f_1
+\Psi^-\circ (N_1^-)^{-1} \circ
R \circ \Theta^{-} \circ  \Lb_1 f_2\\
& \qquad \quad + \Psi^-_D\circ \Psi_b^{-1}\circ
\Lb_{b1}\circ R\circ\Psi^{-1}f_1 \\
& \qquad = - \frac{3}{2\sqrt{2}}
\Psi^-\circ (N_1^-)^{-1} \circ
R f_1
+\Psi^-_D\circ \Lb_{b1}\circ 
 \Psi_b^{-1}\circ R\circ \Psi^{-1}f
\end{align*}
modulo
 $\Psi^{-1}_b\circ R\circ \Psi^{-1}f$

The relation 
\eqref{bndrycndn-} can be read
as
\begin{align*}
\left(\frac{1}{\sqrt{2}} N^-_1
-iT^0\right)u_b^{1,-} +&\Psi^0_b u_b^{1,-}
 -\Lb_1u_b^{2,-}
+ \frac{1}{\sqrt{2}} N^-_2u_b^{2,-}\\ 
 & =
- \frac{3}{2\sqrt{2}}
\Psi^-\circ (N_1^-)^{-1} \circ
R f_1
+\Psi^-_D\circ \Lb_{b1}\circ 
\Psi_b^{-1}\circ\Psi^{-1}f
,
\end{align*}
modulo
$\Psi^{-1}_b\circ R\circ \Psi^{-1}f$.

We set
\begin{equation}
\label{defnu1-}
u_b^{1,-} := 0 
\end{equation}
and thus we have to choose  
$u_b^{2,-}$ which satisfies
\begin{align}
\nonumber
 \Lb_1u_b^{2,-} -& \frac{1}{\sqrt{2}}
   N^-_2u_b^{2,-}\\
\label{u2-Eqn}
&= - \frac{3}{2\sqrt{2}}
\Psi^-\circ (N_1^-)^{-1} \circ
R f_1
+\Psi^-_D\circ \Lb_{b1}\circ 
\Psi_b^{-1}\circ\Psi^{-1}f
\end{align}
modulo 
$\Psi^{-1}_b u_b^2$
and 
$\Psi_b^{-1}\circ R\circ \Psi^{-1} f$.

As $f$ is \dbar-closed,
there exists a $\phi\in L^2(\Omega)$ such that $\mdbar \phi = f$
as in \cite{Hor65}, and in
particular
we have
\begin{equation*}
\Lb_1 \phi_b = R f_1.
\end{equation*}
Thus, for $f\in W^{s}_{(0,1)}(\Omega)\cap \mbox{ker}(\mdbar)$, the
condition $R\circ f_1\in A_b^{s-1/2}(\partial\Omega)$ is satisfied
 (for $s-1/2<0$ we can use
  Equation 2.6 of \cite{MS01} in place of
 the Sobolev Trace Theorem to conclude
  $R\circ f_1 \in W^{s-1/2}(\partial\Omega)$;
 see \eqref{f1Est} below)
and according to the hypothesis on the regularity of $\mdbar_b$,
we can find a $\phi_b'\in W^{s-1/2}(\partial\Omega)$ such that
\begin{equation}
\label{defnphib}
\Lb_{b1} \phi_b' = R f_1.
\end{equation}
Furthermore, we have
\begin{equation*}
\Lb_{b1}  (\phi_b')^{\psi^-} =   
 R f_1^{\psi^-}
+ \Psi^0_0\phi_b'  ,
\end{equation*}
where $\Psi^0_0 = [\Lb_{b1},\Psi^-]$ is a zero order operator
which has a symbol such that the projection of the support of which
onto the second (transform)
component
is contained in $\lrc^0$ (and in fact has
strictly positive distance to the part of the boundary
$\partial\lrc^0 \cap  \{-\frac{3}{4}|\xi_{1,2}| = \xi_3\}$).
In general, we write $\Psi^k_0$ to denote an operator of order
$k$ whose symbol is such that the projection of its support
onto the second (transform)
component
is contained in $\lrc^0$.

We now commute the $\Lb_1$ operator through
 the first term on the right of
\eqref{u2-Eqn}:
\begin{align*}
 - \frac{3}{2\sqrt{2}}
 \Psi^-\circ \left(N_1^-\right)^{-1} \circ
 R f_1
=&  -\frac{3}{2\sqrt{2}}
 \left(N_1^-\right)^{-1} \circ
R f_1^{\psi^-}
\\
= & - \frac{3}{2\sqrt{2}}
\left(N_1^-\right)^{-1} \circ
\Lb_{b1} (\phi_b')^{\psi^-}
 + \Psi^{-1}_0 \phi_b'
\\
=& - \frac{3}{2\sqrt{2}} \Lb_{b1}
\circ \left(N_1^-\right)^{-1} \circ
 (\phi_b')^{\psi^-}\\
&- \frac{3}{2\sqrt{2}}
 \left[ 
   \left(N_1^-\right)^{-1}, \Lb_{b1}
 \right](\phi_b')^{\psi^-}+ \Psi^{-1}_0 \phi_b',
\end{align*}
modulo $\Psi_b^{-2}\circ R \circ f_1$.

The condition for $u_b^{2,-}$,
 given by \eqref{u2-Eqn}, becomes
\begin{align}
\nonumber
 \Lb_{b1}u_b^{2,-} - \frac{1}{\sqrt{2}}
&N^-_2u_b^{2,-} =
 - \frac{3}{2\sqrt{2}} \Lb_{b1}
\circ \left(N_1^-\right)^{-1} \circ
(\phi_b')^{\psi^-}\\
\nonumber
&
- \frac{3}{2\sqrt{2}}
\left[ 
\left(N_1^-\right)^{-1}, \Lb_{b1}
\right](\phi_b')^{\psi^-}
+\Psi^-_D\circ \Lb_{b1}\circ 
\Psi_b^{-1}\circ\Psi^{-1}f
\\
\label{u2-cndn}
&\qquad+ \Psi^{-1}_0 \phi_b'
+\Psi_b^{-1} \circ R\circ
\Psi^{-1} f+\Psi_b^{-2}\circ R \circ f_1,
\end{align}
modulo 
$\Psi^{-1}_b u_b^2$
.
This suggests we set
\begin{equation}
\label{u2set}
 u_{b}^{2,-}
  := -\frac{3}{2\sqrt{2}}
   \left(N_1^-\right)^{-1} 
  (\phi_b')^{\psi^-}
  +
   \Psi^-_D\circ 
   \Psi_b^{-1}\circ\Psi^{-1}f,
\end{equation}
where the second term is the explicit operator
 given in 
\eqref{u2-Eqn}.
With this choice of $u_b^{2,-}$ we compute,
 with the help of Proposition \ref{propN2N1},
\begin{align*}
 \frac{1}{\sqrt{2}}
  N^-_2u_b^{2,-} =&
   -\frac{3}{4}
   N^-_2\circ \left(N_1^-\right)^{-1} 
   (\phi_b')^{\psi^-} + \Psi^{-1}_b\circ
    R\circ \Psi^{-1}f \\
   =& \frac{3}{2\sqrt{2}} 
    \left[ 
    \left(N_1^-\right)^{-1}, \Lb_{b1}
    \right]
   (\phi_b')^{\psi^-}  
   +\Psi^{-2}_b \circ 
 \Lb_{b1} (\phi_b')^{\psi^-}
   + \Psi^{-1}_b\circ
   R\circ \Psi^{-1}f \\
=& \frac{3}{2\sqrt{2}} 
\left[ 
\left(N_1^-\right)^{-1}, \Lb_{b1}
\right]
(\phi_b')^{\psi^-}  
+\Psi^{-2}_b \circ 
R f_1
+ \Psi^{-1}_b\circ
R\circ \Psi^{-1}f\\
& +\Psi^{-2}_b \phi_b'.
\end{align*}

We thus have with the choice
 \eqref{u2set}
\begin{align*}
\Lb_{b1}u_b^{2,-} - \frac{1}{\sqrt{2}}
N^-_2u_b^{2,-} =&
- \frac{3}{2\sqrt{2}} \Lb_{b1}
\circ \left(N_1^-\right)^{-1} \circ
(\phi_b')^{\psi^-}\\
&-\frac{3}{2\sqrt{2}} 
\left[ 
\left(N_1^-\right)^{-1}, \Lb_{b1}
\right]
(\phi_b')^{\psi^-}  \\
&\Lb_{b1}
\circ \Psi^-_D\circ 
\Psi_b^{-1}\circ\Psi^{-1}f
+ \Psi^{-2}_b \phi_b'
\\
&
+\Psi_b^{-2} \circ R f_1
+\Psi_b^{-1} \circ R\circ
\Psi^{-1} f
\end{align*}
which is what was
desired, modulo an error term
$\Psi^{-1}_0\phi_b'$, which will 
handled by the choice of 
$u_b^{0}$.

We now turn to the boundary equations involving $u_b^{j,+}$ for
$j=1,2$, and look to solve
\begin{equation}
\label{bndrycndn+}
\left(\frac{1}{\sqrt{2}} N^-_1
-iT^0\right)u_b^{1,+} +\Psi^0_b u_b^{+} -\Lb_1u_b^{2,+}
= -\frac{2}{\sqrt{2}}
\left( R\circ \Theta^{-} f_1\right)^{\psi^+}
\end{equation}
modulo error terms involving $f$.

In $\widetilde{\lrc}^+$ we have
\begin{align*}
\sigma_1 \left( \frac{1}{\sqrt{2}} N_1^-
-iT^0 \right) =& \frac{1}{\sqrt{2}} |\Xi(x,\xi)| + \xi_3\\
\gtrsim& |\xi|,
\end{align*}
and since there exists a $c>0$ such that $\xi_3>c$ in $\mbox{supp
} \psi^+$, we can find a type of inverse to the operator
$\frac{1}{\sqrt{2}} N_1^- -iT^0 $.
With this in mind we define the symbol
\begin{align*}
\alpha^{\psi^+_D}(x,\xi) =& \frac{\psi^+_D(\xi)}{\sigma_1\left( \frac{1}{\sqrt{2}} N_1^-
	-iT^0  \right)}\\
=& \frac{\psi^+_D(\xi)} {
	\frac{1}{\sqrt{2}} |\Xi(x,\xi)| +
	\xi_3}
,
\end{align*}
where $\psi^+_D$ is defined in analogy to
$\psi^-_D$.  Namely, $\psi^+_D$ has the
properties
$\psi^+_D(\xi)\in C^{\infty}(\widetilde{\lrc}^+)$,
$\psi^+_D(\xi)= \psi^+_D(\xi/|\xi|)$ for $|\xi|\ge 1$,
and such that
$\psi^+_D \equiv 1$ on $\mbox{supp } \psi^+$.  Also, the restriction to
the disc, $\psi^+_D\big|_{\{|\xi|\le 1\}}$, has relatively compact support in
in the interior of $\widetilde{\lrc}^+$.

Then
the composition of operators
\begin{equation*}
\left(  \frac{1}{\sqrt{2}} N_1^-
-iT^0  \right) \circ Op(\alpha^{\psi^+_D})
\end{equation*}
has as symbol
\begin{align*}
\sigma\left[ \left(  \frac{1}{\sqrt{2}} N_1^-
-iT^0 \right) \circ Op(\alpha^{\psi^+_D})\right]=&
\sigma\left(  \frac{1}{\sqrt{2}} N_1^-
-iT^0  \right) \sigma(\alpha^{\psi^+_D})   \\
=&
\left( \frac{1}{\sqrt{2}} |\Xi(x,\xi)| +
\xi_3 \right)
\frac{\psi^+_D(\xi)}{\frac{1}{\sqrt{2}} |\Xi(x,\xi)| + \xi_3}\\
=& \psi^+_D(\xi)
\end{align*}
modulo $\mathcal{S}^{-1}(\partial\Omega)$.  Furthermore, the
same calculations give
\begin{equation*}
\left( \frac{1}{\sqrt{2}} N_1^-
-iT^0  \right) \circ  \Psi^+\circ Op(\alpha^{\psi^+_D})
=\Psi^+
\end{equation*}
modulo $\Psi^{-1}(\partial\Omega)$.

We thus choose $u_b^{1,+}$ according to
\begin{equation}
\label{1+}
u_b^{1,+} =   \left[ Op(\alpha^{\psi^+_D})\left(-\frac{2}{\sqrt{2}}R\circ
\Theta^-
f_1
\right)\right]^{\psi^+}.
\end{equation}
Then, from above,
\begin{align*}
\left( \frac{1}{\sqrt{2}} N_1^-
-iT^0  \right)u_b^{1,+}
= & \left(\frac{1}{\sqrt{2}} N_1^-
-iT^0  \right) \circ  \Psi^+\circ Op(\alpha^{\psi^+_D})
\left(-\frac{2}{\sqrt{2}}R\circ \Theta^-
f_1
\right) \\
=&  -\frac{2}{\sqrt{2}} \left(R\circ \Theta^-
f_1\right)^{\psi^+} + \Psi_b^{-1}\circ R\circ\Psi^{-1} f
.
\end{align*}

Then with  $u_b^{1,+}$ according to \eqref{1+} and with
$u_{b}^{2,+}=0$, \eqref{bndrycndn+} is satisfied, modulo error
terms of the form $ \Psi_b^{-1}\circ R\circ\Psi^{-1} f$.

It remains to choose $u_b^0$, for which we recall has to include a
contribution to handle the error term, $\Psi^{-1}_0\phi_b'$
arising in
the construction of $u_b^{j,-}$ in \eqref{u2-cndn}.  In the region
$\lrc^0$ we can invert the operator $\Lb_1$ in a similar way we
dealt with $\frac{1}{\sqrt{2}} N_1^- -iT^0 $ in
$\widetilde{\lrc}^+$ since
\begin{equation*}
\sigma(\Lb_1) \gtrsim |\xi_1 + i \xi_2| \gtrsim |\xi|.
\end{equation*}
Our choice for $u_{b}^{1,0}$ and $u_{b}^{2,0}$ is analogous
(but reversed) to the case of $u_{b}^{1,+}$ and $u_{b}^{2,+}$
above. Namely, we take
$u_{b}^{1,0}=0$ and $u_{b}^{2,0}$ to be given by
\begin{equation}
\label{20}
u_{b}^{2,0}: =  Op(\beta^{\psi^0_D})\circ \left(\frac{2}{\sqrt{2}}R\circ
\Theta^-
f_1\right)^{\psi^0}
+ Op(\beta^{\psi^0_D})\circ \Psi^{-1}_0\phi_b',
\end{equation}
where
\begin{equation*}
\beta^{\psi^0_D}(x,\xi) =
\frac{\psi^0_D(\xi)}{\sigma(\Lb_1)},
\end{equation*}
and $\psi^0_D(\xi)$ is defined so that
$\psi^0_D(\xi)\in C^{\infty}(\lrc^0)$ and
$\psi^0_D\equiv 1$ on $\mbox{supp }\psi^0$,
with the additional condition that $\psi^0_D\equiv 1$
on the projection of $\mbox{supp}\left(\sigma \left(\Psi^{-1}_0\right)\right)$ onto
$\lrc^0$ (here we make use of the assumption outlined in
Section \ref{notn} that $\psi^-\equiv 1$ in a neighborhood of
$\lrc^-\cap (\lrc^0)^c$).

With $u_{b}^{1,0}$ and $u_{b}^{2,0}$ so chosen, we have
\begin{equation*}
\left( \frac{1}{\sqrt{2}} N^-_1
-iT^0\right)u_{b}^{1,0} +
\Psi^0_b u_{b}^{0} - \Lb_1 u_{b}^{2,0}
=  -\frac{2}{\sqrt{2}}\left( R\circ \Theta^-
f_1\right)^{\psi^0}-\Psi^{-1}_0\phi_b'
\end{equation*}
modulo error terms of the form $\Psi_b^{-1}\circ
R\circ \Psi^{-1} f$, and
$\Psi^{-2}_b \phi_b'$.

The solution to \eqref{bndrycndn} now comes from
\begin{equation*}
u_b^{j} = u_b^{j,-} + u_b^{j,0} + u_b^{j,+}
\end{equation*}
for $j=1,2$.  From above, we have the properties:
\begin{align}
\label{defnu1}
u_b^1=&
\Psi^-\circ\left(N^-_1\right)^{-1}\circ
R\circ \Theta^{-} f_1 +  \left[ Op(\alpha^{\psi^+_D})\left(-\frac{2}{\sqrt{2}}R\circ
\Theta^-
f_1
\right)\right]^{\psi^+}\\
\nonumber
=& \Psi_b^{-1}\circ R\circ\Psi^{-1} f_1
\end{align}
and
\begin{align}
\label{defnu2}
u_b^2=& -\frac{1}{\sqrt{2}}\Psi^-\circ
\left(N^-\right)^{-1}
\phi_b'-
 \Psi^-\circ\left(N^-\right)^{-1}\circ
  R\circ\Theta^-
f_2\\
\nonumber
&+  Op(\beta^{\psi^0_D})\circ \left(\frac{2}{\sqrt{2}}R\circ \Theta
f_1\right)^{\psi^0}
+ Op(\beta^{\psi^0_D})\circ \Psi^{-1}_0\phi_b'\\
\nonumber
=& \Psi^{-1}_b\phi_b' +\Psi_b^{-1}\circ R\circ\Psi^{-1} f.
\end{align}

Furthermore, on the boundary
\begin{align*}
\left(\frac{1}{\sqrt{2}} N^-_1
-iT^0\right)u_b^1 -s_0 u_b^1 -\Lb_{b1}u_b^2
+ \frac{1}{\sqrt{2}} N^-_2u_b^2 = -\frac{2}{\sqrt{2}}
R\circ \Theta^{-} f_1
,
\end{align*}
modulo
\begin{equation*}
\Psi^{-2}_b \phi_b'
 +\Psi_b^{-2} \circ R f_1
 +\Psi_b^{-1} \circ R\circ
 \Psi^{-1} f.
\end{equation*}

Before we handle estimates we recall a definition 
 we made in \cite{Eh18_halfPlanes} which classified 
some of the pseudodifferential operators which arise
 in this article:
\begin{defn}
	We say an operator $B\in \Psi^{-k}$
	for $k\ge 1$ is
	{\it decomposable} if for any $N\ge k$ 
	it can be written in the form 
	\begin{equation*} 
	B = A_{-k} + \Psi^{-N},
	\end{equation*}
	where $A_{-k}\in \Psi^{-k}$ is an operator 
	satisfying the condition
 that the symbol,
$\sigma(A)(x,\rho,\xi,\eta)$, 
is meromorphic (in $\eta$) with poles at 
\begin{equation*}
\eta=q_1(x,\rho,\xi), \ldots, q_k(x,\rho,\xi)
\end{equation*}	
with $q_i(x,\rho,\xi)$ themselves symbols of 
pseudodifferential operators of
order 1 (restricted to $\eta=0$),
 and with the imaginary parts of the poles,
 $q_i(x,\rho,\xi)$ being elliptic operators,
such that for each 
$\rho$, $\mbox{Res}_{\eta=q_i}
\sigma(A) \in \mathcal{S}^{k+1}
(\mathbb{R}^n)$ with symbol estimates
uniform in the $\rho$ parameter. 
\end{defn}

For such decomposable operators we will use the
 following estimates
(see Theorems 2.3 and 2.5 in \cite{Eh18_halfPlanes}):
\begin{thrm}
\label{thrmSobIntEst}
	Let $f\in W^s(\Omega)$ for 
	$s\ge 0$.  Let
	$A\in \Psi^{-k}(\mathbb{R}^{4})$, $k\ge 1$ be a
decomposable operator.  Then
\begin{equation*}
	\|A f\|_{W^s(\Omega)} \lesssim \|
	f\|_{W^{s-k}(\Omega)}.
\end{equation*}
\end{thrm}
Note that these estimates are not immediate,
 as we consider a function supported on the
domain, $\Omega$, to be extended to all of 
 $\mathbb{R}^4$ when it is the argument of a
pseudodifferential operator. 

All pseudodifferential operators above of
 the form $\Psi^{-k}$ for $k\ge 1$
are decomposable as the arise from the inverses to 
differential elliptic operators.
 
We have the following estimates for our
 boundary solution:
\begin{prop}
\label{propEstub}
With $u_b = u_b^1 \omegab_1 +
 u_b^2 \omegab_2$, and
$u_b^1$ and $u_b^2$ defined according to 
 \eqref{defnu1} and
 \eqref{defnu2}, we have
\begin{equation*}
 \|u_b\|_{W^{s+1/2}(\partial\Omega)}
  \lesssim \|f\|_{W^{s}(\Omega)}
\end{equation*}
for $s\ge 0$.
\end{prop}
\begin{proof}
For $u_b^1$ defined as in \eqref{defnu1} we have
	estimates
\begin{align*}
	\|u_b^1\|_{W^{s+1/2}(\partial\Omega)} \lesssim&
	\| \Psi_b^{-1}\circ R\circ\Psi^{-1} f_1\|_{W^{s+1/2}(\partial \Omega)}\\
	\lesssim&
	\|  R\circ\Psi^{-1} f_1\|_{W^{s}(\partial \Omega)}\\
	\lesssim&
	\| \Psi^{-1}f\|_{W^{s+1/2}( \Omega)}\\
	\lesssim&
	\| f\|_{W^{s-1/2}( \Omega)}.
\end{align*}
The estimates moving from the
	boundary to the whole domain in the
	third step above generally work with
	the hypothesis that $s$ is strictly
	greater than $0$.  With a little
	extra effort (using that
the $\Psi^{-1}$ operator comes from
 $\Theta^-$ and thus defines a solution
  to an elliptic equation), 
 the estimates can be generalized 
	to the case $s\ge 0$. 
In the last step we used Theorem \ref{thrmSobIntEst}
 for the decomposable $\Psi^{-1}$ operator.	 

In estimating $u_b^{2,-}$ 
 we will use \cite{MS01}
 (see Equation 2.6 of the article)
for estimates involving $f_1$, the 
coefficient of the component 
orthogonal to $\mdbar\rho$.
In particular,
\begin{align}
\nonumber
 \| f_1 \|_{W^{s-1/2}(\partial\Omega)}
  \lesssim& \|f\|_{W^{s}(\Omega)}
  + \|\mdbar f\|_{W^{s}(\Omega)}\\
\label{f1Est}
 \lesssim& \|f\|_{W^{s}(\Omega)}
\end{align}
for $s\ge 0$.  Thus,
for $u_b^2$ defined as in \eqref{defnu2} we have estimates
\begin{align*}
	\|u_b^2\|_{W^{s+1/2}(\partial\Omega)} \lesssim&
	\left\| \Psi_b^{-1}
	\phi_b'+   \Psi_b^{-1}\circ R\circ\Psi^{-1} f
	\right\|_{W^{s+1/2}(\partial\Omega)}\\
	\lesssim&
	\| \phi_b'\|_{W^{s-1/2}(\partial\Omega)}
	+  \left\| \Psi_b^{-1}\circ R\circ\Psi^{-1} f
	\right\|_{W^{s+1/2}(\partial\Omega)}\\
	\lesssim&
	\| Rf_1\|_{W^{s-1/2}(\partial\Omega)}
	+  \left\|  f
	\right\|_{W^{s-1/2+\epsilon}(\Omega)}\\
	\lesssim&
	\| f\|_{W^{s}(\Omega)},
	\end{align*}
	where $\epsilon$ is a small positive
	number.  
\end{proof}

\section{Estimates for the \dbar-problem}
We now obtain estimates on our solution.
\begin{thrm}
	\label{allestu}
	Let $u$ be defined by \eqref{solnpg},
	\eqref{defnu1}, and \eqref{defnu2}.  Then $u$ satisfies
	\begin{equation*}
	\square u = f  \qquad \mbox{on }\Omega,
	\end{equation*}
	modulo smooth terms,
	with the boundary relation
\begin{equation*}
\mdbar u \rfloor \mdbar\rho
 \Big|_{\partial\Omega}
=
	 R\circ
	\Psi^{-2} f+
	\Psi_b^{-2}\circ R\circ \Psi^0
	f
	+  \Psi_b^{-1}\circ R\circ \Psi^{-1} f +\Psi^{-2}_b
	\phi_b',
	\label{bndryref}
\end{equation*}
modulo smooth terms, denoted
 $R_b^{-\infty}$, which can be 
estimated according to 
\begin{equation}
\label{estSmooth}
 \|R_b^{-\infty}\|_{W^s(\partial\Omega)}
  \lesssim \| u_b \|_{L^2(\partial\Omega)}
\end{equation}
for all $s\ge 0$, and
	where $\phi_b'$ is defined as in \eqref{defnphib}
	
	 Furthermore, we
	have the estimates
\begin{equation*}
	\label{estu}
	 \|u \|_{W^{s+1}(\Omega)} \lesssim \|f\|_{W^{s}(\Omega)}
\end{equation*}
	for $s\ge 0$.
\end{thrm}
\begin{proof}

We recall $u$ as defined by
	\eqref{solnpg}:
\begin{equation*}
	u=G(2f)+
	P(u_b).
\end{equation*}
	We can use the estimate 
from Proposition \ref{propEstub} in
	Theorems \ref{poissest} and \ref{greenest} to estimate the terms
	$G(2f) + P(u_b)$,
	leading to
\begin{align*}
	\| u\|_{W^{s}(\Omega)} \lesssim &
	\| G(2f) + P(u_b) \|_{W^s(\Omega)}
	\\
	\lesssim &
	\| f\|_{W^{s-2}(\Omega)} + \|u_b \|_{W^{s-1/2}(\partial\Omega)}\\
	\lesssim &
	\| f\|_{W^{s-2}(\Omega)} + \|f \|_{W^{s-1}(\Omega)}.
\end{align*}
\end{proof}

We can now 
 construct a solution to the equation $\mdbar
\phi =f $ with $f$ a $(0,1)$-form and prove our Main Theorem.  The
form, $f$, in this section will therefore satisfy the compatibility
condition $\mdbar f=0$.
We prove the
\begin{thrm}
	\label{solnthrm}
	Let $\Omega\subset \mathbb{C}^2$ be a smoothly bounded
	pseudoconvex domain.
	Let $f\in W^{s}_{(0,1)}(\Omega)$ such that $\mdbar f =0$.
	Suppose there exists a solution operator,
	$K_b$ such that $\mdbar_b K_b g = g$ for $g\in
	A^0_b(\partial\Omega)$ and
	$K_b: A^s_b(\partial\Omega) \rightarrow
	W^{s}(\Omega)$ for all $s\ge -1/2$.
	Then there exists a solution operator, $K$, such that
	\begin{equation*}
	\mdbar K f = f
	\end{equation*}
	with the continuity property $K: W^{s}_{(0,1)}(\Omega)\cap
	\mbox{ker}(\mdbar)
	\rightarrow W^{s}(\Omega)$
 for all $s\ge 0$.
\end{thrm}

We base our construction of the solution operator on our solution
to the boundary value problem
\begin{equation}
\label{bvpint}
\square u = f \qquad \mbox{on }\Omega,
\end{equation}
with the boundary condition
\begin{equation}
\overline{L}_2 u_1 -s_0 u_1-\Lb_1u_2=
 R\circ
\Psi^{-2} f+
 \Psi_b^{-2}\circ R f
+  \Psi_b^{-1}\circ R\circ \Psi^{-1} f  +\Psi^{-2}_b
\phi_b',
\label{dbarbndry}
\end{equation}
 modulo smooth terms estimated 
by \eqref{estSmooth},
with $\phi_b'$ given by \eqref{defnphib}.
Theorem \ref{allestu} gave estimates of our chosen solution.
In addition we prove estimates for $\mdbar u$:
\begin{lemma}
	\label{dbaruest}
	\begin{equation*}
	\|\mdbar u\|_{W^{s+2}(\Omega)} \lesssim \|f\|_{W^{s}(\Omega)}.
	\end{equation*}
\end{lemma}
\begin{proof}

On the boundary $\mdbar u$ has the property
\begin{equation}
	\label{dpr}
\mdbar u\big|_{\partial\Omega} =  R\circ
	\Psi^{-2} f
+\Psi_b^{-2}\circ R f_1
	+  \Psi_b^{-1}\circ R\circ \Psi^{-1} f
	+\Psi^{-2}_b
	\phi_b'
\end{equation}
modulo smooth terms, by Theorem \ref{allestu}.
 Furthermore, we have the estimates
\begin{equation*}
	\|\mdbar
	 u\big|_{\partial\Omega}
\|_{W^{s+3/2}(\partial\Omega)}
	\lesssim \|f\|_{W^{s}(\Omega)}
\end{equation*}
	which follows from investigating each term on the right-hand side
	of \eqref{dpr}, as
well as the estimates of the smooth terms
 from \eqref{estSmooth}.  Terms of the form
	$ \Psi_b^{-1}\circ R\circ \Psi^{-1} f $ are handled as in Proposition
	 \ref{propEstub}.
	Furthermore, we show
\begin{align*}
	\| \Psi_b^{-2}\circ R  f_1\|_{W^{s+3/2}(\partial\Omega)}
	\lesssim & \|R  f_1\|_{W^{s-1/2}(\partial\Omega)}\\
	\lesssim & \|f\|_{W^{s}(\Omega)},
\end{align*}
and with Theorem \ref{thrmSobIntEst},
\begin{align*}
	\| R\circ
	\Psi^{-2} f\|_{W^{s+3/2}(\partial\Omega)}
	\lesssim & \| \Psi^{-2} f\|_{W^{s+2}(\Omega)}\\
	\lesssim & \|f\|_{W^{s}(\Omega)},
\end{align*}
	and
	\begin{align*}
	\|  \Psi^{-2}_b  \phi_b' \|_{W^{s+3/2}(\partial\Omega)}
	\lesssim&    \|  \phi_b' \|_{W^{s-1/2}(\partial\Omega)}\\
	\lesssim&    \|  R  f_1 \|_{W^{s-1/2}(\partial\Omega)}\\
	\lesssim&    \|   f \|_{W^{s}(\Omega)}.
	\end{align*}

	As $u$ is a solution to
\begin{equation}
	\label{deluv}
	\square u = f
\end{equation}
	and as $ f$ is \dbar-closed, we can apply
	$\mdbar$ to both sides of \eqref{deluv} to obtain
\begin{align*}
	0=&\mdbar\square u\\
	=& \mdbar\vartheta\mdbar u
\end{align*}
	i.e.,
\begin{align*}
	\mdbar\vartheta\mdbar u   =& 0
\end{align*}
	with a Dirichlet condition with respect to $\mdbar u
	\big|_{\partial\Omega}$ given above. In terms of a Green's
	operator and Poisson's operator (on the level of $(0,2)$-forms
	with respect to the operator $\mdbar\vartheta$; we denote these
	operators $G^2$ and $P^2$, respectively) we have
\begin{equation*}
	\mdbar u = G^2 \left(0\right) + P^2
	\left(\mdbar u \big|_{\partial\Omega}\right).
\end{equation*}
	Theorems \ref{greenest} and \ref{poissest}, or rather the case
	relating a combination of the Theorems in which estimates for the
	solution $v=G^2(g)+P^2(v_b)$
	to the boundary value problem
	\begin{equation*}
	\mdbar\vartheta v = g
	\end{equation*}
	with boundary value $v\big|_{\partial\Omega}=v_b$
	are given as
\begin{equation*}
	\|v\|_{W^{s+2}(\Omega)}
	\lesssim \| g\|_{W^{s}(\Omega)} +
	\|v_b\|_{W^{s+3/2}(\partial\Omega)},
\end{equation*}
	lead to the estimates
\begin{align*}
	\| \mdbar u \|_{W^{s+2}(\Omega)}
	\lesssim & \left\|\mdbar u \big|_{\partial\Omega}
	\right\|_{W^{s+3/2}(\partial\Omega)}\\
	\lesssim & \|f\|_{W^{s}(\Omega)}
\end{align*}
from the boundary 
	 relation in Theorem \ref{allestu}.
\end{proof}

\begin{proof}[Proof of Theorem \ref{solnthrm}]
	We first consider
\begin{align}
	\mdbar \left(\vartheta u \right)
	\nonumber
	=&\square u -\vartheta\mdbar u\\
	\label{dbar1}
	=& f -\vartheta\mdbar u,
\end{align}
	modulo smooth terms.
	The term $\vartheta\mdbar u$ can be estimated by Lemma
	\ref{dbaruest}.
	We let the operator
	$S:W^k(\Omega) \rightarrow W^{k-\delta}(\Omega)$ 
	(for all $\delta>0$),
	$k\ge 1$, be the
	linear solution operator 
	of Sibony-Straube to
\begin{equation}
	\label{dv}
	\mdbar v= \vartheta\mdbar u
\end{equation}
	(see Theorem 5.3
	in \cite{St}) 
Note that from
	\eqref{dbar1} it follows that
	$\vartheta\mdbar u= f-\mdbar \left(\vartheta u \right)$ is \dbar-closed.
Thus, with $v$ defined by
\begin{equation}
	\label{vstr}
	v =S\left(\vartheta\mdbar u \right)
\end{equation}
	we have \eqref{dv},
	and
\begin{align*}
	\|v\|_{W^{s+1-\delta}(\Omega)} =&
	\left\|S \left(
	\vartheta\mdbar u\right)\right\|_{W^{s+1-\delta}(\Omega)}\\
	\lesssim &
	\|\vartheta\mdbar u\|_{W^{s+1}(\Omega)}\\
	\lesssim &\|\mdbar u\|_{W^{s+2}(\Omega)}\\
	\lesssim & \| f\|_{W^{s}(\Omega)} .
\end{align*}

	Then, from \eqref{dbar1}, we have the solution
	$\vartheta u+v$:
	\begin{equation}
	\label{dbareqn}
	\mdbar \left(\vartheta u + v\right) = f
	\end{equation}
	with estimates
	\begin{equation*}
	\| \vartheta u+v \|_{W^{s}(\Omega)} \lesssim \|
	f\|_{W^{s}(\Omega)}.
	\end{equation*}

	To write our solution operator, we recall the operators which
	went into the construction of our
	solution $u$.  The solution $u$ was written
	\begin{equation*}
	u= P(u_b)+G(2f)
	\end{equation*}
	where $u_b$ was chosen via \eqref{defnu1} and \eqref{defnu2}.
	In order to stress the dependence on the data form, $f$,
	we write $u_b^1\omegab_1+u_b^2\omegab_2$ together
	as $U_b(f)$, where $U_b$ represents the operators on the right hand side of the expressions above
	for $u_b^1$ and $u_b^2$.
	Thus, the solution operator, which we define as
$N'$, to the boundary value problem \eqref{bvpint} and
	\eqref{dbarbndry} is given by
	\begin{equation*}
	N'f=
	P \left( U_b(f) \right)
	+ G(2f).
	\end{equation*}
	And finally, the solution operator,
	$K$,
	can be written according to \eqref{dbareqn} as
	\begin{equation*}
	K(f) = \vartheta N'f
	+S\left(f-\mdbar\vartheta N'f\right)
	\end{equation*}
	As $K$ consists of compositions of linear operators, so is $K$
	itself.
\end{proof}

\end{document}